\documentclass[oneside]{amsart}
\usepackage[T1]{fontenc}
\usepackage[latin9]{inputenc}
\usepackage{geometry}
\usepackage{mathtools}
\usepackage{amstext}
\usepackage{amsthm}
\usepackage{amssymb}

\makeatletter
\numberwithin{equation}{section}
\numberwithin{figure}{section}
\theoremstyle{plain}
\newtheorem{thm}{\protect\theoremname}
\theoremstyle{remark}

\theoremstyle{plain}
\newtheorem{lem}[thm]{\protect\lemmaname}

\DeclareFontEncoding{OT2}{}{}
\DeclareFontSubstitution{OT2}{cmr}{m}{l}

\DeclareFontFamily{OT2}{cmr}{\hyphenchar\font45 }
\DeclareFontShape{OT2}{cmr}{m}{l}{%
<5><6><7><8><9>gen*wncyr%
<10><10.95><12><14.4><17.28><20.74><24.88>wncyr10}{}

\DeclareMathAlphabet{\mathcyr}{OT2}{cmr}{m}{l}
\DeclareMathAlphabet{\mathcyb}{OT2}{cmr}{b}{l}
\SetMathAlphabet{\mathcyr}{bold}{OT2}{cmr}{b}{l}

\usepackage{amstext}
\usepackage{amsthm}
\usepackage{amssymb}
\usepackage{ytableau}

\usepackage{shuffle}
\usepackage[all]{xy}

\newcommand{\sha}{\mathbin{\widetilde{\mathcyr{sh}}}}
\newcommand{\sh}{\mathbin{\mathcyr{sh}}}

\DeclareMathOperator{\dep}{dep}

\DeclareMathOperator{\id}{id}
\DeclareMathOperator{\spann}{span}

\makeatother

\providecommand{\lemmaname}{Lemma}
\providecommand{\remarkname}{Remark}
\providecommand{\theoremname}{Theorem}

\begin{document}
\address[Minoru Hirose]{Institute for Advanced Research, Nagoya University,  Furo-cho, Chikusa-ku, Nagoya, 464-8602, Japan}
\email{minoru.hirose@math.nagoya-u.ac.jp}

\address[Hideki Murahara]{The University of Kitakyushu,  4-2-1 Kitagata, Kokuraminami-ku, Kitakyushu, Fukuoka, 802-8577, Japan}
\email{hmurahara@mathformula.page}

\address[Tomokazu Onozuka]{Institute of Mathematics for Industry, Kyushu University 744, Motooka, Nishi-ku, Fukuoka, 819-0395, Japan} \email{t-onozuka@imi.kyushu-u.ac.jp}
\title{On the linear relations among parametrized multiple series}
\author{Minoru Hirose}
\author{Hideki Murahara}
\author{Tomokazu Onozuka}
\begin{abstract}
Parametrized multiple series are generalizations of the multiple zeta
values introduced by Igarashi. In this work, we completely determine
all the linear relations among these parameterized multiple series. Specifically, we prove the
following two statements: the linear part of the Kawashima relation
for multiple zeta values can be generalized to the parametrized multiple
series; any linear relations among the parametrized multiple series
can be written as a linear combination of the linear part of the Kawashima relation. 
\end{abstract}

\maketitle

\section{Introduction}

In \cite{Iga11,Iga12}, Igarashi studied parametrized multiple 
series (PMS), which are defined by 
\[
\zeta^{\textrm{PMS}}(k_{1},\dots,k_{r};\alpha)\coloneqq\sum_{0\le m_{1}<\cdots<m_{r}}\frac{(\alpha)_{m_{1}}}{m_{1}!}\cdot\frac{m_{r}!}{(\alpha)_{m_{r}}}\cdot\frac{1}{(m_{1}+\alpha)^{k_{1}}\cdots(m_{r}+\alpha)^{k_{r}}}\in\mathbb{C}
\]
for positive integers $k_{1},\dots,k_{r}$, with $k_{r}>1$, and a complex number $\alpha$, with $\Re\alpha>0$. 
Here, $(\alpha)_{m}:=\alpha(\alpha+1)\cdots(\alpha+m-1)$
is the Pochhammer symbol. The PMS are generalizations of the multiple
zeta values (MZVs) as $\zeta^{\textrm{PMS}}(k_{1},\dots,k_{r};1)$
is equal to the MZV
\[
\zeta(k_{1},\dots,k_{r})\coloneqq\sum_{0<m_{1}<\cdots<m_{r}}\frac{1}{m_{1}^{k_{1}}\cdots m_{r}^{k_{r}}}\in\mathbb{R}.
\]
Igarashi showed that the PMS satisfy the same relations as the cyclic sum formula \cite{HO03} and Ohno relation
\cite{Ohn99} with respect to MZVs. 
%It is known these two relations  that the linear part of the Kawashima relation includes these two MZV's relations. 
It is known from \cite{Kaw09} and \cite{TW10} that these two MZV relations are included in the so-called \emph{linear part of the Kawashima relation}, which is described as follows. 
Let $\mathfrak{H}\coloneqq\mathbb{Q}\left\langle x,y\right\rangle $
and $z_{k}:=yx^{k-1}$ for a positive integer $k$. 
Use $\mathfrak{H}^{0}:=\mathbb{Q}+y\mathfrak{H}x\subset\mathfrak{H}^{1}:=\mathbb{Q}+y\mathfrak{H}\subset\mathfrak{H}$.
Then, we define a $\mathbb{Q}$-linear map $Z\colon \mathfrak{H}^0\to\mathbb{R}$, 
ring automorphism $\phi:\mathfrak{\mathfrak{H}}\to\mathfrak{H}$, 
and bilinear product $\mathcal{\ast}$ on $\mathfrak{H}^1$ as
\begin{align*}
&Z(z_{k_{1}}\cdots z_{k_{r}}):=\zeta(k_{1},\dots,k_{r}),\\
&\phi(x):=x,\quad\phi(y):=x+y,\\
&1\ast u=u\ast1=u,\\
&z_{k}u\ast z_{l}v:=z_{k}(u\ast z_{l}v)+z_{l}(z_{k}u\ast v)+z_{k+l}(u\ast v).
\end{align*}
Then, the linear part of the Kawashima relation is given by the identity 
\begin{align} \label{kawa}
Z(\phi(w_{1}\ast w_{2})x)=0 \quad (w_{1},w_{2}\in y\mathfrak{H})
\end{align}
which is a special case of the Kawashima relation \cite{Kaw09}.
Now, we define a $\mathbb{Q}$-linear map $L_{\alpha}\colon y\mathfrak{H}x\to\mathbb{C}$ as 
\[
L_\alpha(z_{k_{1}}\cdots z_{k_{r}})
:=\zeta^{\textrm{PMS}}(k_{1},\dots,k_{r};\alpha+1). 
\]
Note that $Z(u)=L_{0}(u)$ for $u\in y\mathfrak{H}x$.
Thus, the first main result of this work is as follows: 
\begin{thm}%[Main theorem 1]
\label{thm:main1}For $w_{1},w_{2}\in y\mathfrak{H}$ and $\alpha\in\mathbb{C}$ with
$\Re\alpha>-1$, we have
\[
L_{\alpha}(\phi(w_{1}\ast w_{2})x)=0.
\]
\end{thm}
Theorem \ref{thm:main1} is a generalization of \eqref{kawa} and Igarashi's results. 
Furthermore, Theorem \ref{thm:main1} exhausts all the linear relations among the PMS, which is the second main result of this work. 
\begin{thm}%[Main theorem 2]
\label{thm:main2}Let $u\in y\mathfrak{H}x$. If $L_{\alpha}(u)=0$
for all $\alpha\in\mathbb{C}$ with $\Re\alpha>-1$, then
\[
u\in{\rm span}_{\mathbb{Q}}\{\phi(w_{1}\ast w_{2})x\mid w_{1},w_{2}\in y\mathfrak{H}\}.
\]
\end{thm}
Summarizing the above two theorems, we obtain the following result.
\begin{thm}
 \[
  \bigcap_{{\rm Re}(\alpha)>-1}\ker(L_{\alpha})={\rm span}_{\mathbb{Q}}\{\phi(w_{1}\ast w_{2})x\mid w_{1},w_{2}\in y\mathfrak{H}\}.
 \]
\end{thm}

\section{Proofs of the main theorems}
For a given path $\gamma:(0,1)\to\mathbb{R}$ and differential $1$-forms
$\omega_{1}(t),\dots,\omega_{k}(t)$, the iterated integral
$I_{\gamma}(\omega_{1},\dots,\omega_{k})$ is defined as
\[
I_{\gamma}(\omega_{1}(t),\dots,\omega_{k}(t))\coloneqq\int_{0<t_{1}<\cdots<t_{k}<1}\prod_{j=1}^{k}\omega_{j}(\gamma(t_{j})).
\]
Furthermore, we substitute $I_{0,1}(-)=I_{\gamma_{0,1}}(-)$ and $I_{0,\infty}(-)=I_{\gamma_{0,\infty}}(-)$,
where $\gamma_{0,1}(t)=t$ and $\gamma_{0,\infty}(t)=t/(t-1)$.
We identify $x\in\mathfrak{H}$ and $y\in\mathfrak{H}$ with the differential
1-forms $dt/t$ and $dt/(1-t)$, respectively. By fixing a complex number
$\alpha$ with $|\alpha|<1$, from \cite[Lemma 2.2]{Iga12}, we
have 
\begin{align}
L_{\alpha}(yu_{1}\cdots u_{k}x)=I_{0,1}((t/(1-t))^{\alpha}y,u_{1},\cdots,u_{k},((1-t)/t)^{\alpha}x),\label{eq1}
\end{align}
where $u_{1},\dots,u_{k}\in\{x,y\}.$ By a change of variables, (\ref{eq1})
is equal to
\[
I_{0,\infty}((-t)^{\alpha}\phi(y),\phi(u_{1}),\cdots,\phi(u_{k}),(-t)^{-\alpha}\phi(x)).
\]
Thus, we have
\begin{align}
L_{\alpha}(\phi(yu_{1}\cdots u_{k})x) & =I_{0,\infty}((-t)^{\alpha}y,u_{1},\cdots,u_{k},(-t)^{-\alpha}\phi(x))\nonumber \\
 & =\int_{0<t_{0}<t_{1}<\cdots<t_{k}<t_{k+1}<1}\left(\frac{\gamma(t_{0})}{\gamma(t_{k+1})}\right)^{\alpha}y(t_{0})\phi(x)(t_{k+1})\prod_{j=1}^{k}u_{j}(\gamma(t_{j}))\ \ \ (\gamma\coloneqq\gamma_{0,\infty})\nonumber \\
 & =\sum_{m=0}^{\infty}(-\alpha)^{m}\int_{\substack{0<t_{0}<t_{1}<\cdots<t_{k}<t_{k+1}<1\\
t_{0}<s_{1}<\cdots<s_{m}<t_{k+1}
}
}\frac{d\gamma(s_{1})\cdots d\gamma(s_{m})}{\gamma(s_{1})\cdots\gamma(s_{m})}y(t_{0})\phi(x)(t_{k+1})\prod_{j=1}^{k}u_{j}(\gamma(t_{j}))\nonumber \\
 & =\sum_{m=0}^{\infty}(-\alpha)^{m}I_{0,\infty}(y,\{u_{1},\dots,u_{k}\}\sh\{x\}^{m},\phi(x)),\label{eq:2-1}
\end{align}
where $\sh$ is the shuffle product and $\{x\}^{m}$ indicates the $m$
times repetition of $x$; the last expression equals 
\[
\sum_{m=0}^{\infty}(-\alpha)^{m}\sum_{f:\{1,\dots,k\}\hookrightarrow\{1,\dots,k+m\}}I_{0,\infty}(y,u'_{1},\dots,u'_{k+m},\phi(x)),
\]
where
\[
u'_{i}=\begin{cases}
u_{j} & \text{if there exists \ensuremath{j} such that \ensuremath{f(j)=i},}\\
x & \text{{\rm otherwise}.}
\end{cases}
\]
For a non-negative integer $m$, we define a linear map $\sigma_{m}\colon \mathfrak{H}^1\to \mathfrak{H}^1$
as
\begin{align*}
\sigma_{m}(1) & :=\delta_{m,0},\\
\sigma_{m}(yw) & :=y(w\sh x^{m}).
\end{align*}
In other words, we have $\sigma_{m}(z_{k_{1}}\cdots z_{k_{r}})=(-1)^{m}\sum_{e_{1}+\cdots+e_{r}=m}z_{k_{1}+e_{1}}\cdots z_{k_{r}+e_{r}}\prod_{j=1}^{r}\binom{k_{j}+e_{j}-1}{e_{j}}$.
Then, from (\ref{eq:2-1}), we obtain
\begin{align}
L_{\alpha}(\phi(w)x) & =\sum_{m=0}^{\infty}(-\alpha)^{m}L_{0}(\phi(\sigma_{m}(w))x)\label{eq:3}
\end{align}
for $w\in y\mathfrak{H}.$

Let $S_{1}$ be an automorphism on $\mathfrak{H}$ defined by $S_{1}(x)=x$ and $S_{1}(y)=x+y$; let $S$ be a $\mathbb{Q}$-linear map from $\mathfrak{H}^1$ to $\mathfrak{H}^1$
defined by $S(1)=1$ and $S(yw)=yS_{1}(w)$; let $\beta_{1}$ be an automorphism on $\mathfrak{H}$ defined by $\beta_{1}(x)=y$ and $\beta_{1}(y)=x$, 
and let $\beta$ be a $\mathbb{Q}$-linear map from $\mathfrak{H}^1$ to $\mathfrak{H}^1$ 
defined by $\beta(1)=1$ and $\beta(yw)=y\beta_{1}(w)$. 
We denote as $Z_{\ast}\colon \mathfrak{H}^1\to\mathbb{R}$ the
unique map characterized by the following properties: $Z_{\ast}|_{\mathfrak{H}^0}=Z$,
$Z_{\ast}(u\ast v)=Z_{\ast}(u)Z_{\ast}(v)$, and $Z_{\ast}(y)=0$.
Put $Z_{\ast}^{\star}:=Z_{\ast}\circ S:\mathfrak{H}^{1}\to\mathbb{R}$. 
We also define $Z_{\ast}^{(\alpha)},Z_{\ast}^{\star,(\alpha)}\colon \mathfrak{H}^1\to\mathbb{R}[[\alpha]]$
as 
\begin{align*}
Z_{\ast}^{(\alpha)}(w) & :=\sum_{m\ge0}(-\alpha)^{m}Z_{\ast}(\sigma_{m}(w)),\\
Z_{\ast}^{\star,(\alpha)}(w) & :=\sum_{m\ge0}(-\alpha)^{m}Z_{\ast}^{\star}(\sigma_{m}(w))=Z_{\ast}^{(\alpha)}(S(w)).
\end{align*}
Note that
\begin{align*}
Z_{\ast}^{(\alpha)}(z_{k_{1}}\cdots z_{k_{r}})
& =\sum_{0<m_{1}<\cdots<m_{r}}\frac{1}{(m_{1}+\alpha)^{k_{1}}\cdots(m_{r}+\alpha)^{k_{r}}}\in\mathbb{C},\\
Z_{\ast}^{\star,(\alpha)}(z_{k_{1}}\cdots z_{k_{r}})
& =\sum_{0<m_{1}\le\cdots\le m_{r}}\frac{1}{(m_{1}+\alpha)^{k_{1}}\cdots(m_{r}+\alpha)^{k_{r}}}\in\mathbb{C}
\end{align*}
for positive integers $k_{1},\dots,k_{r}$, with $k_{r}>1$, and complex
number $\alpha$, with $\Re\alpha>-1$.
\begin{lem}
\label{lem:KYX}For positive integers $k_{1},\dots,k_{r}$, we have
\begin{align*}
L_{0}(S\beta(z_{k_{1}}\cdots z_{k_{r}})x) & =\sum_{i=0}^{r}(-1)^{r-i+1}Z_{\ast}^{\star}(z_{k_{1}}\cdots z_{k_{i}})Z_{\ast}(\sigma_{1}(z_{k_{r}}\cdots z_{k_{i+1}})).
\end{align*}
\end{lem}
\begin{proof}
The coefficient of $T^{0}x$ on the left-hand side of the equality in \cite[Theorem 3.1]{KXY21} is equal to $L_{0}(S\beta(z_{k_{1}}\cdots z_{k_{r}})x)$; the right-hand
side of the equality in \cite[Theorem 3.1]{KXY21} is equal to
\begin{align*}
 \sum_{m=0}^{\infty}(-x)^{m}\sum_{i=0}^{r}(-1)^{r-i}Z_{\ast}^{\star}(z_{k_{1}}\cdots z_{k_{i}})Z_{\ast}(\sigma_{m}(z_{k_{r}}\cdots z_{k_{i+1}})).
\end{align*}
Thus, the lemma follows from \cite[Theorem 3.1]{KXY21}. 
\end{proof}
Let $d$ be an automorphism on $\mathfrak{H}$ defined by $d(x)=x$, $d(y)=-y$.
Put $\tilde{S}:=S\circ d\colon\mathfrak{H}^1\to \mathfrak{H}^1$. 
Note that the map $\tilde{S}$ is a $\ast$-homomorphism.
We define a linear map $\Delta:\mathfrak{H}^{1}\to\mathfrak{H}^{1}\otimes\mathfrak{H}^{1}$
by $\Delta(z_{k_{1}}\cdots z_{k_{r}})=\sum_{i=0}^{r}z_{k_{1}}\cdots z_{k_{i}}\otimes z_{k_{r}}\cdots z_{k_{i+1}}$.
Note that $\Delta(z_{k_{1}}\cdots z_{k_{r}})$ is not $\sum_{i=0}^{r}z_{k_{1}}\cdots z_{k_{i}}\otimes z_{k_{i+1}}\cdots z_{k_{r}}$. 
\begin{lem}
\label{lem:key}For $w=z_{k_{1}}\cdots z_{k_{r}}\in y\mathfrak{H}$, we have
%and $\alpha\in\mathbb{C}$ with $\Re\alpha>-1$
\[
L_{\alpha}(\phi(w)x)=\sum_{i=0}^{r-1}Z_{\ast}^{(\alpha)}(z_{k_{1}}\cdots z_{k_{i}})Z_{\ast}^{(\alpha)}(\tilde{S}\sigma_{1}(z_{k_{r}}\cdots z_{k_{i+1}})).
\]
\end{lem}
\begin{proof}
We first prove the case for $\alpha=0$. By Lemma \ref{lem:KYX}, for
$u\in y\mathfrak{H}$, we have
\[
Z(S\beta d(u)x)=-Z(f(u)),
\]
where $f\colon y\mathfrak{H}\to y\mathfrak{H}$ is a linear map defined
by 
\[
f(z_{k_{1}}\cdots z_{k_{r}})=\sum_{i=0}^{r}\tilde{S}(z_{k_{1}}\cdots z_{k_{i}})\ast\sigma_{1}(z_{k_{r}}\cdots z_{k_{i+1}}).
\]
Then, the left-hand side is calculated as 
\begin{align*}
L_{0}(\phi(w)x) & =-Z(S\beta d\tilde{S}(w)x)\\
 & =Z(f(\tilde{S}(w))).
\end{align*}
 Here, since $f=(\ast)\circ(\tilde{S}\otimes\sigma_{1})\circ\Delta$
and $\Delta\circ\tilde{S}=(\tilde{S}\otimes\tilde{S})\circ\Delta$,
we have
\begin{align*}
f\tilde{S} & =(\ast)\circ(\tilde{S}\otimes\sigma_{1})\circ\Delta\circ\tilde{S}\\
 & =(\ast)\circ(\tilde{S}\otimes\sigma_{1})\circ(\tilde{S}\otimes\tilde{S})\circ\Delta\\
 & =(\ast)\circ(\id\otimes\sigma_{1}\tilde{S})\circ\Delta\\
 & =(\ast)\circ(\id\otimes\tilde{S}\sigma_{1})\circ\Delta.
\end{align*}
Hence, we have 
\[
L_{0}(\phi(w)x)=\sum_{i=0}^{r-1}Z(z_{k_{1}}\cdots z_{k_{i}})Z(\tilde{S}\sigma_{1}(z_{k_{r}}\cdots z_{k_{i+1}}))
\]
and obtain the result for $\alpha=0$. 

From (\ref{eq:3}) and the result for $\alpha=0$, we have
\begin{align*}
L_{\alpha}(\phi(w)x) & =\sum_{m=0}^{\infty}(-\alpha)^{m}L_{0}(\phi(\sigma_{m}(w))x)\\
 & =\sum_{m=0}^{\infty}(-\alpha)^{m}(Z\otimes Z\tilde{S}\sigma_{1})\circ\Delta\circ\sigma_{m}(w)\\
 & =\sum_{m=0}^{\infty}\sum_{m_{1}+m_{2}=m}(-\alpha)^{m}(Z\otimes Z\tilde{S}\sigma_{1})\circ(\sigma_{m_{1}}\otimes\sigma_{m_{2}})\circ\Delta(w)\\
 & =\sum_{m_{1}=0}^{\infty}\sum_{m_{2}=0}^{\infty}(-\alpha)^{m_{1}+m_{2}}(Z\sigma_{m_{1}}\otimes Z\tilde{S}\sigma_{1}\sigma_{m_{2}})\circ\Delta(w)\\
 & =(Z_{\ast}^{(\alpha)}\otimes Z_{\ast}^{(\alpha)}\tilde{S}\sigma_{1})\circ\Delta(w)\\
 & =\sum_{i=0}^{r-1}Z_{\ast}^{(\alpha)}(z_{k_{1}}\cdots z_{k_{i}})Z_{\ast}^{(\alpha)}(\tilde{S}\sigma_{1}(z_{k_{r}}\cdots z_{k_{i+1}})).
\end{align*}
Thus, the lemma is proved. 
\end{proof}
We now define a linear map $\psi\colon y\mathfrak{H}\to y\mathfrak{H}$ as
\[
\psi(z_{k_{1}}\cdots z_{k_{r}}):=\sum_{i=0}^{r-1}z_{k_{1}}\cdots z_{k_{i}}\ast\tilde{S}(\sigma_{1}(z_{k_{r}}\cdots z_{k_{i+1}})).
\]
\begin{lem}
We have $\psi(y\mathfrak{H})\subset y\mathfrak{H}x$. 
\end{lem}
\begin{proof}
We prove this lemma using the theory of quasisymmetric functions \cite{Hof15}.
Let $\Phi$ be a linear map from $\mathfrak{H}^1$ to the
space of quasisymmetric functions given by 
\[
\Phi(z_{k_{1}}\cdots z_{k_{r}})\coloneqq\sum_{0<m_{1}<\cdots<m_{r}}t_{m_{1}}^{k_{1}}\cdots t_{m_{r}}^{k_{r}}\in
%\mathbb{Q}[[t_1,t_2,\dots]]
\lim_{\substack{\longleftarrow\\n}} \mathbb{Q}[t_{1},\dots,t_{n}].
\]
Then, $\Phi$ is a bijection \cite[Theorem 2.2]{Hof15}. By definition,
we have
\begin{align}
 & \Phi(\psi(z_{k_{1}}\cdots z_{k_{r}}))\nonumber \\
 & =\sum_{i=0}^{r-1}(-1)^{r-i}\sum_{j=i+1}^{r}k_{j}\Phi(z_{k_{1}}\cdots z_{k_{i}}\ast S(z_{k_{r}}\cdots z_{k_{j+1}}z_{k_{j}+1}z_{k_{j-1}}\cdots z_{k_{i+1}}))\nonumber \\
 & =\sum_{j=1}^{r}k_{j}\sum_{i=0}^{j-1}(-1)^{r-i}\Phi(z_{k_{1}}\cdots z_{k_{i}}\ast S(z_{k_{r}}\cdots z_{k_{j+1}}z_{k_{j}+1}z_{k_{j-1}}\cdots z_{k_{i+1}}))\nonumber \\
 & =\sum_{j=1}^{r}k_{j}\sum_{i=0}^{j-1}(-1)^{r-i}\sum_{\substack{0<m_{1}<\cdots<m_{i}\\
0<m_{r}\leq\cdots\leq m_{i+1}
}
}t_{m_{1}}^{k_{1}}\cdots t_{m_{r}}^{k_{r}}\cdot t_{m_{j}}\nonumber \\
 & =\sum_{j=1}^{r}k_{j}\sum_{i=0}^{j-1}(-1)^{r-i}\left(\sum_{\substack{0<m_{1}<\cdots<m_{i}<m_{i+1}\\
0<m_{r}\leq\cdots\leq m_{i+1}
}
}+\sum_{\substack{0<m_{1}<\cdots<m_{i}\\
0<m_{r}\leq\cdots\leq m_{i+1}\leq m_{i}
}
}\right)t_{m_{1}}^{k_{1}}\cdots t_{m_{r}}^{k_{r}}\cdot t_{m_{j}}\ \ \ (m_{0}:=0)\nonumber \\
 & =\sum_{j=1}^{r}k_{j}(-1)^{r-j+1}\sum_{\substack{0<m_{1}<\cdots<m_{j}\\
0<m_{r}\leq\cdots\leq m_{j}
}
}t_{m_{1}}^{k_{1}}\cdots t_{m_{r}}^{k_{r}}\cdot t_{m_{j}}\label{eq:ttt}\\
 & \in\Phi(y\mathfrak{H}x).\nonumber 
\end{align}
Thus, the lemma is proved.
\end{proof}
From Lemma \ref{lem:key}, we have
\begin{align*}
L_{\alpha}(\phi(w)x) & =Z_{\ast}^{(\alpha)}(\psi(w))\qquad(w\in y\mathfrak{H}).
\end{align*}
Since $Z_{\ast}^{(\alpha)}$ is injective on $\mathfrak{H}^0$ (see \cite{Joy12,Bou14}), $L_{\alpha}(\phi(w)x)=0$
if and only if $\psi(w)=0$. Thus, for the proofs of Theorems 2 and
3, it is sufficient to show that $\ker\psi=y\mathfrak{H}\ast y\mathfrak{H}.$ 
\begin{lem}
\label{lem:gyaku}$\ker\psi\supset y\mathfrak{H}\ast y\mathfrak{H}.$
\end{lem}
\begin{proof}
For $u,v\in y\mathfrak{H}$, we need to show that $\psi(u\ast v)=0$. 
We define a linear map
$\lambda\colon\mathfrak{H}^{1}\to\mathfrak{H}^{1}[[\alpha]]$ by 
\[
\lambda(z_{k_{1}}\cdots z_{k_{r}})=\sum_{m=0}^{\infty}\alpha^{m}\sum_{i=0}^{r}z_{k_{1}}\cdots z_{k_{i}}\ast\tilde{S}(\sigma_{m}(z_{k_{r}}\cdots z_{k_{i+1}}))
\]
and another linear map $\sigma\colon\mathfrak{H}^{1}\to\mathfrak{H}^{1}[[\alpha]]$
as $\sigma(w):=\sum_{m=0}^{\infty}\alpha^{m}\sigma_{m}(w)$. Then,
we have $\lambda=(\ast)\circ(\id\otimes\tilde{S}\sigma)\circ\Delta.$
Since the maps $\Delta$, $\sigma$, and $\tilde{S}$ are $\ast$-homomorphisms,
$\lambda$ is also a $\ast$-homomorphism. Thus, we have $\lambda(u\ast v)=\lambda(u)\ast\lambda(v)$.
Since the coefficients of $\alpha^{0}$ in $\lambda(u)$ and $\lambda(v)$
are $0$,
\[
\psi(u\ast v)=(\textrm{the coefficient of }\alpha\textrm{ in }\lambda(u\ast v))=0. \qedhere
\] 
\end{proof}
We define a bilinear product $\mathcal{\sha }$ on
$\mathfrak{H}^1$ as 
\begin{align*}
1\sha u & =u\sha 1=u,\\
z_{k}u\sha z_{l}v:= & z_{k}(u\sha z_{l}v)+z_{l}(z_{k}u\sha v).
\end{align*}
Then, we define a linear map $\overline{\psi}\colon y\mathfrak{H}\to y\mathfrak{H}$ as
\[
\overline{\psi}(z_{k_{1}}\cdots z_{k_{r}}):=\sum_{i=0}^{r-1}(-1)^{r-i}z_{k_{1}}\cdots z_{k_{i}}\sha \sigma_{1}(z_{k_{r}}\cdots z_{k_{i+1}}).
\]
Set $H_{d}:=\bigoplus_{r=0}^{d}\bigoplus_{k_{1},\dots,k_{r}\ge1}\mathbb{Q}\cdot z_{k_{1}}\cdots z_{k_{r}}\subset\mathfrak{H}^{1}$;
by definition, for $u\in H_{r}$ and $v\in H_{s}$, we have $u\ast v-u\sha v\in H_{r+s-1}.$
We now define a linear operator $\rho$ on $y\mathfrak{H}$ as\\
\[
\rho(z_{k_{1}}\cdots z_{k_{r}}):=\sum_{i=0}^{r-1}(-1)^{r-i}(z_{k_{1}}\cdots z_{k_{i-1}}\sha z_{k_{r}}\cdots z_{k_{i+1}})z_{k_{i}}.
\]

\begin{lem}
\label{lem:.shuffle-1}$\ker\rho\subset y\mathfrak{H}\sha y\mathfrak{H}$.
\end{lem}
\begin{proof}
Let $Y$ be the set of noncommutative polynomials of $X_{1},X_{2},\dots$
without constant terms. 
We define a bilinear map $\langle,\rangle:Y\otimes y\mathfrak{H}\to\mathbb{Q}$
as 
\[
\langle X_{k_{1}}\cdots X_{k_{r}},z_{l_{1}}\cdots z_{l_{s}}\rangle\coloneqq\begin{cases}
1 & (k_{1},\dots,k_{r})=(l_{1},\dots,l_{s}),\\
0 & {\rm \text{otherwise.}}
\end{cases}
\]
Let $L\subset Y$ be the set of Lie polynomials of $X_{1},X_{2},\dots$.
Then, we have \\
\begin{align} \label{eeqq1}
L & =\spann_{\mathbb{Q}}\{[X_{k_{1}},[X_{k_{2}},[\dots,[X_{k_{r-1}},X_{k_{r}}]\dots]]]\mid r\ge1,k_{1},\dots,k_{r}\ge1\}
\end{align}
and
\begin{align} \label{eeqq2}
y\mathfrak{H}\sha y\mathfrak{H}=\{u\in y\mathfrak{H}\mid\langle a,u\rangle=0\text{ for all }a\in L\}
\end{align}
(see \cite[Section 1]{Reu93}). 
Given any $r\ge1$ and $k_{1},\dots,k_{r}\ge1$,
we have 
\begin{align*}
&[X_{k_{1}},[X_{k_{2}},[\dots,[X_{k_{r-1}},X_{k_{r}}]\dots]]] \\
&=\sum_{i=1}^{r}(-1)^{r-i}\sum_{\substack{\{a_{1},\dots,a_{i-1}\}\sqcup\{b_{1},\dots,b_{r-i}\}=\{1,\dots,r-1\}\\
a_{1}<\cdots<a_{i-1}\\
b_{1}<\cdots<b_{r-i}
}
}X_{k_{a_{1}}}\cdots X_{k_{a_{i-1}}}X_{k_{r}}X_{k_{b_{r-i}}}\cdots X_{k_{b_{1}}}\\
 & =\sum_{i=1}^{r}(-1)^{r-i}\sum_{\substack{\sigma:\{1,\dots,r\}\to\{1,\dots,r\}\\
\sigma:{\rm bijection}\\
\sigma(1)<\cdots<\sigma(i-1)<\sigma(i)>\sigma(i+1)>\cdots>\sigma(r)
}
}X_{k_{\sigma(1)}}\cdots X_{k_{\sigma(r)}}
\end{align*}
by definition. Since
\[
\langle X_{k_{\sigma(1)}}\cdots X_{k_{\sigma(r)}},z_{l_{1}}\cdots z_{l_{r}}\rangle=\prod_{j=1}^{r}\delta_{k_{\sigma(j)},l_{j}}=\prod_{j=1}^{r}\delta_{k_{j},l_{\sigma^{-1}(j)}}=\langle X_{k_{1}}\cdots X_{k_{r}},z_{l_{\sigma^{-1}(1)}}\cdots z_{l_{\sigma^{-1}(r)}}\rangle,
\]
we have 
\begin{align*}
 & \langle[X_{k_{1}},[X_{k_{2}},[\dots,[X_{k_{r-1}},X_{k_{r}}]\dots]]],z_{l_{1}}\cdots z_{l_{r}}\rangle\\
 & =\sum_{i=1}^{r}(-1)^{r-i}\sum_{\substack{\sigma:\{1,\dots,r\}\to\{1,\dots,r\}\\
\sigma:{\rm bijection}\\
\sigma(1)<\cdots<\sigma(i-1)<\sigma(i)>\sigma(i+1)>\cdots>\sigma(r)
}
}\langle X_{k_{1}}\cdots X_{k_{r}},z_{l_{\sigma^{-1}(1)}}\cdots z_{l_{\sigma^{-1}(r)}}\rangle\\
 & =\sum_{i=1}^{r}(-1)^{r-i}\langle X_{k_{1}}\cdots X_{k_{r}},\sum_{\substack{\sigma:\{1,\dots,r\}\to\{1,\dots,r\}\\
\sigma:{\rm bijection}\\
\sigma(1)<\cdots<\sigma(i-1)<\sigma(i)>\sigma(i+1)>\cdots>\sigma(r)
}
}z_{l_{\sigma^{-1}(1)}}\cdots z_{l_{\sigma^{-1}(r)}}\rangle\\
 & =\sum_{i=1}^{r}(-1)^{r-i}\langle X_{k_{1}}\cdots X_{k_{r}},(z_{l_{1}}\cdots z_{l_{i-1}}\sha z_{l_{r}}\cdots z_{l_{i+1}})z_{l_{i}}\rangle\\
 & =\langle X_{k_{1}}\cdots X_{k_{r}},\rho(z_{l_{1}}\cdots z_{l_{r}})\rangle.
\end{align*}
Hence, for $u\in\ker\rho$, we have 
\[
\langle[X_{k_{1}},[X_{k_{2}},[\dots,[X_{k_{r-1}},X_{k_{r}}]\dots]]],u\rangle=\langle X_{k_{1}}\cdots X_{k_{r}},\rho(u)\rangle=0
\]
for any $r\ge1$ and $k_{1},\dots,k_{r}\ge1$, which implies that $u\in y\mathfrak{H}\sha y\mathfrak{H}$ by \eqref{eeqq1} and \eqref{eeqq2}. 
\end{proof}
\begin{lem}
\label{lem:.shuffle}$\ker\overline{\psi}\subset y\mathfrak{H}\sha y\mathfrak{H}$.
\end{lem}
\begin{proof}
Let $u\in\ker\overline{\psi}$. From (\ref{eq:ttt}), we have
\[
\psi(z_{k_{1}}\cdots z_{k_{r}})\equiv\sum_{i=1}^{r}(-1)^{r-i+1}k_{i}(z_{k_{1}}\cdots z_{k_{i-1}}\sha z_{k_{r}}\cdots z_{k_{i+1}})z_{k_{i}+1} 
\quad \bmod(H_{r-1}).
\]
Then, we have
\[
\overline{\psi}(z_{k_{1}}\cdots z_{k_{r}})=\sum_{i=1}^{r}(-1)^{r-i+1}k_{i}(z_{k_{1}}\cdots z_{k_{i-1}}\sha z_{k_{r}}\cdots z_{k_{i+1}})z_{k_{i}+1}.
\]
Hence, $\rho=\iota\circ\overline{\psi}$, where $\iota$ is a linear
operator on $y\mathfrak{H}$ defined by 
\[
\iota(z_{k_{1}}\cdots z_{k_{r}}):=k_{r}z_{k_{1}}\cdots z_{k_{r-1}}z_{k_{r}+1}.
\]
Thus, $\ker\overline{\psi}\subset\ker\rho$. 
\end{proof}
\begin{lem}
\label{lem:harmonic}$\ker\psi\subset y\mathfrak{H}\ast y\mathfrak{H}.$
\end{lem}
\begin{proof}
We denote the depth ($=$maximal number of $y$ in the
monomials) of $w$ by $\dep(w)$ and prove that $w\in y\mathfrak{H}\ast y\mathfrak{H}$
for $w\in\ker\psi$ by induction on $\dep(w)$. Let $w$ be a nonzero
element of $\ker\psi$. Then, by Lemma \ref{lem:.shuffle}, there
exist $u_{1},\dots,u_{n},v_{1},\dots,v_{n}\in y\mathfrak{H}$ with
$\dep(u_{i})+\dep(v_{i})=\dep(w)\quad(i=1,\dots,n)$ such that 
\begin{align*}
w -u_{1}\sha v_{1}-\cdots-u_{n}\sha v_{n}\in H_{\dep(w)-1}.
\end{align*}
Since the depth of $u_{i}\sha v_{i}-u_{i}\ast v_{i}$
is smaller than $\dep(w)$, the depth of 
\begin{align*}
w': & =w-u_{1}\ast v_{1}-\cdots-u_{n}\ast v_{n}
\end{align*}
is smaller than the depth of $w$. By Lemma \ref{lem:gyaku}, $w'\in\ker\psi$;
thus, $w'\in y\mathfrak{H}\ast y\mathfrak{H}$ by the induction hypothesis.
Hence, we have $w\in y\mathfrak{H}\ast y\mathfrak{H}$.
\end{proof}
By Lemmas \ref{lem:gyaku} and \ref{lem:harmonic}, we have $\ker\psi=y\mathfrak{H}\ast y\mathfrak{H}.$
Thus, we obtain Theorems \ref{thm:main1} and \ref{thm:main2}.

\section*{Acknowledgment}
This work was supported by JSPS KAKENHI Grant Numbers JP18K13392 and JP19K14511.

\end{document}